\pdfoutput=1 % for arXiv to use pdflatex instead of latex+dvipdf

\documentclass[a4paper,twoside,reqno]{article}

\usepackage{a4wide}
\usepackage{amssymb}
\usepackage{amsmath}
\usepackage{mathtools}
\usepackage{amsthm}

\usepackage{mathrsfs}

\usepackage{enumitem}
\setlist[enumerate,1]{label=(\alph*)}
\setlist[enumerate,2]{label=(\roman*)}

\usepackage{url}
\usepackage[hidelinks]{hyperref}

\usepackage[usenames,dvipsnames]{xcolor}
\usepackage{mathtools}
\mathtoolsset{showonlyrefs}

\usepackage{dsfont}

\theoremstyle{definition}

\theoremstyle{remark}

\swapnumbers
\theoremstyle{plain}

\newtheorem{theorem}{Theorem}[section]
\newtheorem*{theorem*}{Theorem}

\newtheorem{lemma}[theorem]{Lemma}
\newtheorem{corollary}[theorem]{Corollary}

\theoremstyle{remark}
\newtheorem{remark}[theorem]{Remark}
\newtheorem*{remark*}{Remark}

\theoremstyle{definition}
\newtheorem{definition}[theorem]{Definition}
\newtheorem{miniremark}[theorem]{}

\usepackage{authoraftertitle}

\usepackage[only,llbracket,rrbracket]{stmaryrd}

\newcommand{\End}[1]{ \mathrm{End}({#1}) }

\newcommand{\ccspace}[1]{\mathscr{K}(#1)}

\newcommand{\ograss}[2]{\mathbf{G}_0(#1,#2)}

\newcommand{\PC}[1]{\mathbf{P}_{#1}}

\newcommand{\IPC}[1]{\mathscr{P}_{#1}}

\newcommand{\Curr}[1]{\mathscr{D}_{#1}}

\newcommand{\Forms}[1]{\mathscr{D}^{#1}}

\newcommand{\Mass}{\mathbf{M}}

\newcommand{\sphere}[1]{\mathbb{S}^{#1}}

\ifdefined\textint
    \renewcommand{\textint}[2]{{\textstyle\int_{#1}^{#2}}}
\else    
    \newcommand{\textint}[2]{{\textstyle\int_{#1}^{#2}}}
\fi

\ifdefined\textsum
    \renewcommand{\textsum}[2]{{\textstyle\sum_{#1}^{#2}}}
\else  
    \newcommand{\textsum}[2]{{\textstyle\sum_{#1}^{#2}}}
\fi

\newcommand{\natp}{\mathscr{P}}

\newcommand{\integers}{\mathbf{Z}}

\newcommand{\R}{\mathbf{R}}

\newcommand{\Q}{\mathbf{Q}}

\newcommand{\CF}[1]{\boldsymbol{\chi}_{#1}}

\newcommand{\HM}{\mathscr{H}}

\newcommand{\restrict}{ \mathop{ \rule[1pt]{.5pt}{6pt} \rule[1pt]{4pt}{0.5pt} }\nolimits }

\newcommand{\ud}{\ensuremath{\,\mathrm{d}}}

\newcommand{\uD}{\ensuremath{\mathrm{D}}}

\newcommand{\id}[1]{\mathds{1}_{#1}}

\newcommand{\lIm}{[}
\newcommand{\rIm}{]}

\newcommand{\tbwedge}{{\textstyle \bigwedge}}

\newcommand{\tbcup}{{{\textstyle \bigcup}}}

\newcommand{\Dirac}[1]{\boldsymbol{\delta}_{#1}}

\newcommand{\Lbrack}{\llbracket}
\newcommand{\Rbrack}{\rrbracket}

\DeclareMathOperator{\Hom}{Hom}

\DeclareMathOperator{\asssp}{space}

\newcommand{\cnt}[1]{\mathscr{C}^{#1}}

\newcommand{\orthproj}[2]{\mathbf{O}^\ast({#1},{#2})}

\newcommand{\orthgroup}[1]{\mathbf{O}({#1})}

\DeclareMathOperator{\spt}{spt}

\DeclareMathOperator{\conv}{conv}

\DeclareMathOperator{\ray}{ray}

\DeclareMathOperator{\cone}{cone}

\DeclareMathOperator{\Lip}{Lip}

\DeclareMathOperator{\card}{card}

\DeclareMathOperator{\sgn}{sign}

\newcommand{\without}{\!\sim\!}
\DeclareMathOperator{\im}{im}

\newcommand{\XX}{% Space between XX varies depending on where it's used
  \mathchoice% https://tex.stackexchange.com/q/148740/5764
    {\times\mskip-15.5mu \times}% \displaystyle
    {\times\mskip-15.5mu \times}% \textstyle
    {\times\mskip-15.5mu \times}% \scriptstyle
    {\times\mskip-15.5mu \times}% \scriptscriptstyle
}

\DeclareMathOperator{\Conv}{Conv}
\DeclareMathOperator{\p}{\textbf{p}}

\date{\today}

\title{On Polyconvexity and Almgren Uniform Ellipticity
  With Respect to Polyhedral Test Pairs}

\author{
    Maciej Lesniak
}

\hypersetup{
  unicode=true,
  pdfauthor={\MyAuthor},
  pdftitle={\MyTitle},
  pdfsubject={},
  pdfkeywords={},
  pdfproducer={},
  pdfcreator={}
  pdfinfo={
    orcid_ML={0000-0000-0000-0000},
  }
}

\begin{document}

\maketitle

\begin{abstract}
    We study anisotropic geometric energy functionals defined on a class of
    $k$-dimensional surfaces in a~Euclidean space. The classical notion of ellipticity,
    coming from Almgren, for such functionals is investigated. We prove a variant of 
    a recent result of De Rosa, Lei, and Young and show that \emph{uniform} ellipticity 
    of an anisotropic energy functional with respect to real polyhedral chains 
    implies \emph{uniform} polyconvexity of the integrand.
\end{abstract}

\section{Introduction}
\hfill

Let $\ograss{n}{k}$ denote the Grassmannian of oriented $k$-planes in $\R^{n}$. Given
a~continuous integrand $F: \ograss{n}{k} \to (0, \infty)$, one defines \emph{anisotropic
  energy functional} $\Phi_{F}: \PC{k}(\R^{n}) \to \R$, acting on $k$-dimensional
polyhedral chains (see~\S{\ref{section:PolyhedralChains}} for relevant definitions)
in~$\R^{n}$ by
\begin{displaymath}
    \Phi_{F}(T) = \int F(\vec{T}(x))\ud\|T\|(x) \quad \text{for } T \in 
    \PC{k}(\R^{n}),
\end{displaymath}
where $\|T\|$ is a~Radon measure over~$\R^{n}$ and $\vec{T}$ is a~$\|T\|$-measurable
function with values in~$\ograss{n}{k}$.

Assuming $F$ is even (i.e. $F(-\xi) = F(\xi)$ for $\xi \in \ograss nk$) we may consider
$F$ to be defined on the unoriented Grassmannian and then the functional~$\Phi_{F}$ can be
extended to act on varifolds by~setting $\Phi_{F}(V) = \int F(T) \ud V(x,T)$ for
a~$k$-dimensional varifold~$V$ in~$\R^{n}$; see~\cite{Allard1972}. Almgren studied the
regularity of varifolds which are local minimisers of~$\Phi_{F}$ and proved that they are
almost everywhere regular, provided $F$~is \emph{uniformly elliptic} (abbreviated AUE)
\cite[Theorems~1.4 and~1.7]{Almgren1968}. If $F$ is elliptic, then a flat $k$-disc (or
$k$-cube) is the unique minimiser of~$\Phi_{F}$ among all competitors with the same
($k-1$)-dimensional boundary as the disc (or cube). Uniform ellipticity gives additionally
quantitative control on~$\Phi_F$; see~\S{\ref{section:Ellipticity}}.

This notion heavily depends on the choice of competitors. When competitors are rectifiable
currents, any uniformly convex norm on $\tbwedge_k \R^n$ gives rise to a uniformly
elliptic integrand; cf.~\cite[5.1.2]{Federer1969}. Burago and Ivanov
\cite[Theorem~3]{Burago2004} showed that if competitors are rectifiable currents (with
real coefficients), then \emph{semi-ellipticity} is equivalent to being~\emph{extendibly
  convex}, i.e., there exists an extension of the integrand to a~norm on the
whole~$\tbwedge_k \R^n$. Their proof relies on showing, in non-constructive way,
the~existence of a~polyhedral surface with prescribed tangent planes; see~\cite[Theorem~1,
Theorem~2]{Burago2004}.

However, if competitors are allowed only to be \emph{integral} currents, then there exists
a~continuous integrand defined on~$\ograss{4}{2}$ for which the associated energy
functional is elliptic, but the integrand cannot be extended to a~norm on the
whole~$\tbwedge_2 \R^4$; see~\cite[Theorem~5]{Burago2004}.

In~\cite[Theorem~3.2]{Rosa2023}, De Rosa, Lei, and Young provided a~constructive proof
of~\cite[Theorem~2]{Burago2004} and established an~analogous theorem
to~\cite[Theorem~3]{Burago2004} for~Lipschitz multivalued functions, namely, that
ellipticity for multigraphs is equivalent to~\emph{polyconvexity}; see~\cite[Theorem 1.7]{Rosa2023}.

The classical notion of~\emph{polyconvexity}, coming from calculus of variations,
is~defined for functionals acting on functions; cf.~\cite[Chapter~5]{Giusti2003}. In~this
setting integrands are real-valued functions defined on the space of linear
maps~$\Hom(\R^k,\R^n)$ -- for the purposes of the following discussion let us call them
\emph{classical integrands} as opposed to~\emph{geometric integrands} of the type
$\ograss{n}{k} \to \R$; cf.~\cite[5.1.1 and~5.1.9]{Federer1969}.

The~authors of~\cite{Rosa2023} use \emph{polyconvexity} for classical
integrands and briefly describe its connection with being~\emph{extendibly convex} for
related geometric integrands; see \cite[Remark 1.4, Definition 1.6, Definition 1.7]{Rosa2023}.
Later they (as we do as well) call a~geometric integrand~$F$ \emph{polyconvex} in~case it 
is extendibly convex. As a~service to the community, we~present a~detailed formal 
discussion of the relation between~\emph{extendibly convex}
geometric integrands and~\emph{polyconvex} classical integrands
in~section~\ref{section:Polyconvexity}.

A~notion stronger than \emph{polyconvexity} is~\emph{uniform polyconvexity}
(abbreviated~UPC) which involves certain constant~$c > 0$;
see~\ref{def:upolyconvexity}. In this article, we show that \emph{uniform ellipticity} for
polyhedral $k$-chains with constant~$c$ implies \emph{uniform polyconvexity} with any
constant smaller than~$c$. This can be seen as a~refinement of some results contained
in~\cite{Burago2004} and~\cite{Rosa2023}.
\begin{theorem*}[\protect{see~\ref{thm:main} for a precise statement}]
    \label{thm:main_intro}
    Assume $F$ is a geometric integrand satisfying $\Lip F < \infty$. If $F$ is Almgren
    uniformly elliptic with respect to polyhedral $k$-chains with real coefficients and
    with ellipticity constant $0 < c < \infty$, then $F$ is uniformly polyconvex with any
    constant $0 < \tilde{c} < c$.
\end{theorem*}
Our proof is based on \cite[Theorem 3.2]{Rosa2023}, which allows to construct polyhedral
$k$-chains with Gaussian measures close to a given measure and provides control of the
geometry of the constructed surface. Note also that the reverse implication (with
$\tilde{c} = c$) holds by a~simple argument involving the Stokes theorem;
cf.~\cite[5.1.2]{Federer1969}.

This paper is organised as follows. In Section~\ref{section:Preliminaries}, we recall the
basic terminology that will be used throughout the paper. In
Section~\ref{section:Polyconvexity}, we present a formal discussion on the relationship
between the notions of~\emph{polyconvex} and~\emph{extendibly convex} integrands,
demonstrating their equivalence. Moreover, in the same section, we recall the definition
of \emph{uniform polyconvexity} for integrands.  Section~\ref{section:PolyhedralChains} is
devoted to the classical definitions of polyhedral chains and Gaussian measures. Finally,
in~Section~\ref{section:Ellipticity}, we define~\emph{ellipticity} for polyhedral test
pairs and prove the main result.

\section{Preliminaries}
\label{section:Preliminaries}

In~principle we shall follow the notation of Federer;
see~\cite[pp.~669--671]{Federer1969}. In~particular, $A \without B$ is the set-theoretic
difference and $\natp$ denotes the set of positive integers.  Diverting from Federer, if
$-\infty < s < r < \infty$, we~occasionally write $(s,r)$, $[s,r]$, $(s,r]$, and~$[s,r)$
to denote appropriate intervals in~$\R$.  We shall use letters~$k$ and~$n$ to denote two
integers satisfying $0 \le k < n$. We say that a~function $f$ defined on some open subset
of a~Banach space with values in another Banach space is of class~$\cnt{l}$ if it is
$l$-times differentiable and $\uD^l f$ is~continuous; cf~\cite[3.1.11]{Federer1969}.

\begin{definition}[\protect{cf.~\cite[1.3.2]{Federer1969}}]
    By $\Lambda(n, k)$ we denote set of all increasing maps of $\{1,\ldots,k\}$ into
    $\{1,\ldots,n\}$.
\end{definition}

\begin{remark*}
    Note that any $\lambda \in \Lambda(n,k)$ is uniquely determined by $\im \lambda$ which
    is a $k$-element subset of $\{1,2,\ldots,n\}$; thus, choosing
    $\lambda \in \Lambda(n,k)$ corresponds to choosing $k$-elements out of
    $\{1,2,\ldots,n\}$.
\end{remark*}

\begin{definition}
    Let $\xi \in \tbwedge_k \R^n$. We define the vector space associated with
    $\xi$ by
    \begin{displaymath}
        \asssp \xi = \R^n \cap \bigl\{ v : v \wedge \xi = 0 \bigr\} \,.
    \end{displaymath}
\end{definition}

\begin{remark}[\protect{\cite[1.6.1]{Federer1969}}]
    An element $\xi \in \tbwedge_k \R^n$ is simple if and only if
    $\dim \asssp \xi = k$.
\end{remark}

\begin{definition}[\protect{\cite[3.2.28(b)]{Federer1969}}]
    We define the \emph{oriented Grassmannian $\ograss{n}{k}$} as
    \begin{displaymath}
        \ograss{n}{k} = \tbwedge_k \R^n \cap \bigl\{ \xi : \xi \text{ is simple}
        ,\, |\xi| = 1 \bigr\} \,.
    \end{displaymath}
\end{definition}

\begin{remark}
    Let $\xi \in \tbwedge_k \R^n$. Recall from~\cite[1.8.1]{Federer1969} that
    $\|\xi\| = |\xi|$ if and only if $\xi$ is simple.
\end{remark}

\begin{definition}
    Given a vectorspace $X$ we write $\End{X}$ for the space of all linear maps
    mapping~$X$ to~$X$, i.e., $\End{X} = \Hom(X,X)$.
\end{definition}

\begin{definition}[\protect{\cite[1.7.4]{Federer1969}}]
    A linear map $j \in \Hom(\R^k,\R^n)$ is called an~\emph{orthogonal
      injection} if $j^* \circ j = \id{\R^k}$. If $p \in \Hom(\R^n,\R^k)$ and
    $j = p^*$ is an orthogonal injection, then we say that $p$ is
    an~\emph{orthogonal projection}. The set of all orthogonal projections
    $\R^n \to \R^k$ is denoted $\orthproj nk$.
\end{definition}

\begin{remark}
    To each $\lambda \in \Lambda(n, k)$ corresponds the map $\p_{\lambda} \in
    \orthproj nk$ defined as
    \begin{displaymath}
        \p_{\lambda}(x) = (x_{\lambda(1)}, \ldots, x_{\lambda(k)})
        \quad \text{for $x=(x_1,\ldots,x_n) \in \R^{n}$} \,.
    \end{displaymath}
\end{remark}

\begin{definition}[\protect{cf.~\cite[p.~14]{Rockafellar1970}}]
    Let $X$ be a vectorspace and $A \subseteq X$. We define
    \begin{displaymath}
        \ray A = X \cap \{ ta : 0 \le t < \infty ,\, a \in A \}
        \quad \text{and} \quad
        \cone A = \conv (\ray A) \,.
    \end{displaymath}
\end{definition}

\begin{definition}
    Let $X$ be a normed vectorspace and $f : X \to \R$. We say that $f$ is
    \emph{positively homogeneous} if $f(tx) = t f(x)$ whenever $x \in X$ and
    $0 \le t < \infty$.
\end{definition}

\begin{definition}
    Let $X$ be a vectorspace and $f : X \to \R$. We say that $f$ is
    a~\emph{gauge} if it is convex non-negative and positively homogeneous.
    We~say that $f$ is a~\emph{strict gauge} if $f$ is a gauge and
    \begin{displaymath}
        f(x+y) < f(x) + f(y)
        \quad \text{for all $x,y \in X \without \{0\}$ such that $x \notin \ray \{y\}$} \,.
    \end{displaymath}
\end{definition}

\begin{definition}
    A \emph{signed measure} over~$\ograss nk$ is a Daniell integral on the
    lattice of functions~$\ccspace{\ograss nk}$ as defined
    in~\cite[2.5.6]{Federer1969}.
\end{definition}

\begin{definition}[\protect{\cite[2.5.5]{Federer1969}}]
    \emph{Total variation} of a signed measure~$\mu$ over $\ograss nk$ is
    defined as
    \begin{displaymath}
        \|\mu\|_{\mathrm{TV}} := |\mu|(\ograss{n}{k}), \quad \text{where} \quad
        |\mu| = \mu^{+}+\mu^{-} \,.
    \end{displaymath}
\end{definition}

\section{Polyconvexity}
\label{section:Polyconvexity}

\begin{miniremark}
    The notion of \emph{polyconvexity} has been introduced, in the context of calculus of
    variations, for integrands defined on the space of linear maps $\Hom(\R^k,\R^n)$; see,
    e.g.,~\cite[Ch.5, p.~147]{Giusti2003}. In~our case integrands are functions
    on~$\ograss nk \subseteq \tbwedge_k \R^n$ and the corresponding notion of
    polyconvexity can be found, e.g., in~\cite[Definition 1.3]{Rosa2023}. The relation between these
    two notions has been very briefly explained in a~small remark \cite[Remark 1.4]{Rosa2023}. 
    As a~service to the community, as well as for the sake of author's
    own understanding, we present here a~detailed formal discussion of this relation.
\end{miniremark}

\begin{remark}
    \label{rem:minor}
    Assume $f \in \Hom(\R^k,\R^n)$ and $A \in \R^{n \times k}$ is the matrix
    of~$f$ written in the standard bases. For any $\lambda \in \Lambda(n,k)$ the
    matrix of the map $\p_{\lambda} \circ f$ is obtained by choosing the rows
    of~$A$ indexed by~$\lambda$ and $\det(\p_{\lambda} \circ f)$ is
    a~$k \times k$ minor of the matrix~$A$. Having this in mind one readily sees
    that the following two definitions are a special case of the ones given
    by~Giusti.
\end{remark}

\begin{definition}
    Let $f \in \Hom(\R^k,\R^n)$. We define $M(f) \in \R^{\Lambda(n,k)}$ by
    \begin{displaymath}
        M(f)(\lambda) = \det(\p_{\lambda} \circ f)
        \quad \text{for $\lambda \in \Lambda(n,k)$} \,.
    \end{displaymath}
\end{definition}

\begin{definition}[\protect{cf.~\cite[Ch.5, p.~147]{Giusti2003}}]
    \label{def:Giusti_polyconvex}
    A function $G : \Hom(\R^k,\R^n) \to [0,\infty)$ is called \emph{(strictly)
      polyconvex} if there exists a (strict) gauge
    $h : \R^{\Lambda(n,k)} \to [0,\infty)$ such that $G = h \circ M$.
\end{definition}

\begin{remark}
    Note that our definition differs from~\cite[Ch.5]{Giusti2003} in the following way:
    we enforce~$h$ to be positively homogeneous, which immediately yields that $G$ is 
    positively $k$-homogeneous and determines the~growth of~$G$ at infinity.
\end{remark}

\begin{definition}
    An~\emph{integrand} is a continuous non-negative function $F : \ograss nk \to \R$.
\end{definition}

\begin{remark}
    Note that an integrand $F$ is defined only on the Grassmannian $\ograss nk$. Later it
    will be useful to consider positively homogeneous extensions of~$F$ onto the whole
    of~$\tbwedge_k \R^n$.
\end{remark}

\begin{remark}
    Let $(e_1, \ldots, e_k)$ be the standard basis of $\R^k$,
    $E = e_1 \wedge \cdots \wedge e_k \in \ograss kk$, $\xi \in \ograss nk$,
    $p \in \orthproj nk$ be such that $\tbwedge_k p^* E = \xi$, and
    $A \in \R^{n \times k}$ be the matrix of~$p^*$ in the standard bases. Then the columns
    of~$A$ form an~orthonormal basis of~$\asssp \xi$ which is positively oriented with
    respect to~$\xi$. This explains the connection of the following definition with the
    one given by De~Rosa, Lei, and Young.
\end{remark}

\begin{definition}[\protect{cf.~\cite[1.3]{Rosa2023}}]
    \label{def:DRLY_polyconvex}
    Let $(e_1, \ldots, e_k)$ be the standard basis of $\R^k$. An integrand $F$
    is said to be (strictly) \emph{polyconvex} if there is a~(strict) gauge
    \begin{displaymath}
        h : \R^{\Lambda(n,k)} \to [0, +\infty)
        \quad \text{such that} \quad
        F(\xi) =  h(M(p_{\xi}^*))
    \end{displaymath}
    whenever $\xi \in \ograss nk$ and $p_{\xi} \in \orthproj nk$ satisfies
    $\tbwedge_k p_{\xi}^*(e_1 \wedge \cdots \wedge e_k) = \xi$.
\end{definition}

\begin{remark}
    \label{def_poly_indep}
    Definition \ref{def:DRLY_polyconvex} is independent of the choice
    of~$p_{\xi}$ for $\xi \in \ograss nk$. Indeed, let $\xi \in \ograss nk$ and
    $p_1, p_2 \in \orthproj nk$ be such that
    $\tbwedge_k p_{i}^*(e_1 \wedge \cdots \wedge e_k) = \xi$ for
    $i \in \{1,2\}$. There exists $T \in \orthgroup k$ such that
    $p_1^* = p_2^* \circ T$ and we get
    \begin{displaymath}
        M(p_1^*)(\lambda)
        = \det(\p_{\lambda} \circ p_1^*)
        = \det(\p_{\lambda} \circ p_2^* \circ T)
        = \det(\p_{\lambda} \circ p_2^*) \det T
        = M(p_2^*)(\lambda) \det T
    \end{displaymath}
    for $\lambda \in \Lambda(n,k)$. However, we also know that
    \begin{multline}
        \xi = \tbwedge_k p_{2}^*(e_1 \wedge \cdots \wedge e_k)
        = \tbwedge_k p_{1}^*(e_1 \wedge \cdots \wedge e_k)
        = \tbwedge_k (p_{2}^* \circ T) (e_1 \wedge \cdots \wedge e_k)
        \\
        = \tbwedge_k p_{2}^* \circ \tbwedge_k T (e_1 \wedge \cdots \wedge e_k)
        = \det T \tbwedge_k p_{2}^* (e_1 \wedge \cdots \wedge e_k) \,;
    \end{multline}
    hence, $\det T = 1$.
\end{remark}

\begin{remark}
    Let $e_1,\ldots,e_k \in \R^k$ be the standard basis of $\R^k$,
    $E = e_1 \wedge \cdots \wedge e_k \in \ograss kk$, $\xi \in \ograss nk$, and
    $p \in \orthproj nk$ be such that $\tbwedge_k p^* E = \xi$. In~\ref{def_poly_indep} we
    showed that $M(p)$ does not depend on the particular choice of~$p$ so $M$ could
    actually be defined on~$\ograss nk$. Let $u_1,\ldots,u_n$ be the standard basis
    of~$\R^n$. For $\lambda \in \Lambda(n,k)$ set
    $u_{\lambda} = u_{\lambda(1)} \wedge \cdots \wedge u_{\lambda(k)}$ and note that
    $u_{\lambda} = \tbwedge_k \p_{\lambda}^* E$; thus,
    \begin{displaymath}
        \xi \bullet u_{\lambda}
        = \tbwedge_k p^* E \bullet \tbwedge_k \p_{\lambda}^* E
        = \tbwedge_k (p_{\lambda} \circ p^*) E \bullet E
        = \det (p_{\lambda} \circ p^*)
        = M(p^*)(\lambda) \,.
    \end{displaymath}
\end{remark}

\begin{definition}
    Define $N : \tbwedge_k \R^n \to \R^{\Lambda(n,k)}$ by the formula
    \begin{displaymath}
        N(\xi)(\lambda) = \xi \bullet u_{\lambda}
        \quad \text{for $\xi \in \tbwedge_k \R^n$ and $\lambda \in \Lambda(n,k)$} \,.
    \end{displaymath}
\end{definition}

\begin{remark}
    Note that $N$ is a linear isometry.
\end{remark}

\begin{corollary}
    \label{cor:two_def_polyconvexity}
    Let $F$ be an integrand. The following are equivalent
    \begin{enumerate}
    \item $F$ is (strictly) polyconvex in the sense of
        definition~\ref{def:DRLY_polyconvex}.
    \item There exists of (strict) gauge $h : \R^{\Lambda(n,k)} \to [0,\infty)$ such that
        $F(\xi) = h(N(\xi))$ for $\xi \in \ograss nk$.
    \end{enumerate}
\end{corollary}

\begin{corollary}
    \label{cor:Giusti_vs_DRLY}
    Assume $h : \R^{\Lambda(n,k)} \to [0,\infty)$. The following are equivalent
    \begin{enumerate}
    \item $h \circ M : \Hom(\R^k,\R^n) \to \R$ is (strictly) polyconvex in the
        sense of~\ref{def:Giusti_polyconvex};
    \item $h \circ N : \tbwedge_k \R^n \to \R$ is (strictly) polyconvex in the
        sense of~\ref{def:DRLY_polyconvex};
    \item $h | N \lIm \ograss nk \rIm$ can be extended to a~(strict) gauge
        on~$\R^{\Lambda(n,k)}$.
    \end{enumerate}
\end{corollary}

\begin{remark}
    We could have defined polyconvexity in~\ref{def:Giusti_polyconvex}
    and~\ref{def:DRLY_polyconvex} without enforcing~$h$ to be positively homogeneous (only
    convex) and we would get an analogous corollary to~\ref{cor:Giusti_vs_DRLY}. Observe,
    however, that if $h : \R^{\Lambda(n,k)} \to [0,\infty)$ is only convex, then,
    in~general, it might not be possible to find a~gauge which agrees with~$h$
    on~$N \lIm \ograss nk \rIm$; hence, not assuming $h$ to be positively homogeneous
    in~\ref{def:DRLY_polyconvex} significantly affects the class of polyconvex integrands.

    To support this claim let $\sphere{} = \R^{\Lambda(n,k)} \cap \{ y : |y| = 1 \}$,
    choose $e \in N \lIm \ograss nk \rIm \subseteq \sphere{}$, $1 < r < \infty$, and
    $(1+r)^2/r < p < \infty$ (note that $(1+r)^2/r \ge 4$). Consider the non-negative
    convex function $h : \R^{\Lambda(n,k)} \to [0,\infty)$ given by
    \begin{displaymath}
        h(x) = |x-re|^p
        \quad \text{for $x \in \R^{\Lambda(n,k)}$} \,.
    \end{displaymath}
    Next, define a~non-negative positively homogeneous function
    $f : \R^{\Lambda(n,k)} \to [0,\infty]$ by setting
    \begin{displaymath}
        f(x) = h(x)
        \quad \text{and} \quad
        f(tx) = t f(x)
        \quad \text{for $x \in \sphere{}$ and $0 \le t < \infty$} \,.
    \end{displaymath}
    Fix $x = -e$ and $u \in \sphere{}$ such that $e \perp u$. Choose a~geodesic
    $\gamma : \R \to \sphere{}$ such that $\gamma(0) = x$, $\gamma'(0) = u$ (then,
    necessarily $\gamma''(0) = -x$). Direct computations show
    \begin{gather}
        (f \circ \gamma)''(0)
        = \uD^2f(x)uu - \uD f(x) x
        = \uD^2f(x)uu - f(x)
        = \uD^2f(x)uu - h(x) \,,
        \\
        (f \circ \gamma)''(0)
        = (h \circ \gamma)''(0)
        = \uD^2h(x)uu - \uD h(x) x \,,
        \\
        \uD h(x) x = p |x-re|^{p-2} (x - re) \bullet x \,,
        \\
        \text{and} \quad
        \uD^2h(x) uu = p(p-2) |x-re|^{p-4} \bigl( (x - re) \bullet u \bigr)^2 + p |x-re|^{p-2} \,;
    \end{gather}
    hence, recalling $x = -e$ and $e \perp u$,
    \begin{multline}
        \uD^2f(x)uu
        = \uD^2h(x)uu - \uD h(x) x + h(x)
        = |x-re|^{p-2} \bigl( p - p  (x - re) \bullet x + |x-re|^2 \bigr)
        \\
        = |x-re|^{p-2} \bigl(- pr + (1+r)^2 \bigr) \,.
    \end{multline}
    Since $p > (1+r)^2/r$ we get $\uD^2f(x)uu < 0$ so $f$ is not convex.
\end{remark}

\subsection*{Convex hulls}

\begin{definition}[\protect{cf.~\cite[2.5.19]{Federer1969}}]
    Assume $X$ is a set and $x \in X$. The \emph{Dirac measure} over~$X$ having a~single
    atom at~$x$ is denoted by
    \begin{displaymath}
        \Dirac{x}(A) = 
        \begin{cases}
          1 \text{\quad if \quad} x \in A \\
          0 \text{\quad if \quad} x \notin  A
        \end{cases}
        \quad \text{for $A \subseteq X$} \,.
    \end{displaymath}
\end{definition}

\begin{definition}
    \label{def:convexhull}
    Let $F$ be an integrand. The \emph{convex positively homogeneous hull} of~$F$ is
    defined by
    \begin{multline}
        \Conv_{F}(\xi)
        = \inf \left\{ \int{}{} F \ud \mu
            :
            \begin{gathered}
                N \in \natp ,\,
                m_1, \ldots, m_N \in (0,\infty) ,\,
                \eta_1,\ldots,\eta_n \in \ograss nk ,\,
                \\
                \mu = \textsum{i=1}{N} m_{i}\Dirac{\eta_{i}} ,\,
                \quad
                \xi = \textint{}{} \eta \ud \mu(\eta)
            \end{gathered}
        \right\}
        \\
        \text{for $\xi \in \tbwedge_k \R^n$} \,.
    \end{multline}
\end{definition}

\begin{remark}
    $\Conv_{F}$ is the biggest convex positively homogeneous function (i.e. a~gauge) not
    exceeding~$F$ on~$\ograss nk$ pointwise.
\end{remark}

\begin{lemma}
    \label{lem:conv_lemma}
    An integrand $F$ is polyconvex if and only if $F = \Conv_{F} \vert_{\ograss nk}$.
\end{lemma}

\begin{proof}
    Assume $F$ is polyconvex so that there exists a gauge $h : \R^{\Lambda(n,k)} \to \R$
    such that $F = h \circ N|\ograss nk$. Since $N$ is a linear isometry, we see that
    $h \circ N$ is a gauge on $\tbwedge_k \R^n$ so
    \begin{displaymath}
        F = h \circ N|\ograss nk \le \Conv_{F}|\ograss nk \le F \,.
    \end{displaymath}
    On the other hand, if $\Conv_{F}|\ograss nk = F$, we set $h = \Conv_{F} \circ N^{-1}$
    to get $F = h \circ N|\ograss nk$.
\end{proof}

\subsection*{Uniform polyconvexity}

\begin{definition}
    \label{def:upolyconvexity}
    Let $0 < c < \infty$ and $F$ be an integrand. We say that $F$ is
    \emph{uniformly polyconvex (with constant $c$)} if
    \begin{displaymath}
        \textsum{i=1}{d} m_i F( \eta_{i} ) - F( \eta_0 )
        \ge c \bigl( \textsum{i=1}{d} m_i|\eta_{i}| - |\eta_{0}| \bigr)
    \end{displaymath}
    whenever $d \in \natp$, $m_1, \ldots, m_d \in (0,\infty)$,
    $\eta_{0}, \ldots, \eta_{d} \in \ograss nk$ are such that
    $\eta_{0} = \textsum{i=1}{d} m_i \eta_{i}$.
    \\
    If $F$ is uniformly polyconvex with constant $c$, we shall write
    $F \in \mathrm{UPC(c)}$.
\end{definition}

\begin{remark}
    Assume $F \in \mathrm{UPC(c)}$ for some $0 < c < \infty$, $d \in \natp$,
    $m_1, \ldots, m_d \in (0,\infty)$, $\eta_{0}, \ldots, \eta_{d} \in \ograss nk$, and
    $\eta_{0} = \textsum{i=1}{d} m_i \eta_{i}$. Set
    $\mu = \textsum{i=1}{N} m_{i}\Dirac{\eta_{i}}$. Then we have
    \begin{displaymath}
        \textint{}{} F \ud \mu = \textsum{i=1}{d} m_i F( \eta_{i} ) \ge F( \eta_0 ) \,;
    \end{displaymath}
    hence, $\Conv_{F}(\eta_0) = F(\eta_0)$. Since $\eta_0$ was arbitrary, we see
    from~\ref{lem:conv_lemma} that $F$ is polyconvex. Consequently, we get that
    uniformly polyconvexity implies polyconvexity.
\end{remark}

\section{Polyhedral Chains}
\label{section:PolyhedralChains}

\begin{miniremark}
    For the convenience of the reader and the sake of completeness we recall
    here the classical definitions concerning polyhedral chains.
\end{miniremark}

\begin{definition}[\protect{\cite[4.1.7]{Federer1969}}]
    Let $U$ be an open subset of $\R^{n}$ and $m$ be nonnegative integer.
    By~$\Forms{m}(U)$ we denote set of \emph{differential forms of~degree~$m$}
    and class~$\cnt{\infty}$ on~$U$ with real coefficients, whose support is a~compact
    subset of~$U$. By $\Curr{m}(U)$ we denote set of \emph{$m$-dimensional
      currents} in~$U$, i.e., the space of~continuous linear functionals
    on~$\Forms{m}(U)$.
\end{definition}

\begin{remark}[\protect{\cite[4.1.5 and 4.1.7]{Federer1969}}]
    In case $T \in \Curr{m}(U)$ is \emph{representable by integration}, then
    $T = \|T\| \wedge \vec{T}$, where $\|T\|$ is a~Radon measure over~$U$ and
    $\vec{T}$ is a~$\|T\|$-measurable function with values in~$\tbwedge_m \R^n$.
\end{remark}

\begin{definition}[\protect{\cite[4.1.8]{Federer1969}}]
    Let $A$ and $B$ be open subsets of $\R^{m}$ and $\R^{n}$. Moreover, let
    \begin{displaymath}
        p : A \times B \to A \quad \text{and} \quad q : A \times B \to B
    \end{displaymath}
    be the canonical projections onto the first and second factor respectively.
    For any $S \in \Curr{i}(A)$ and $T \in \Curr{j}(B)$ the
    \emph{Cartesian product} of $S$ and $T$ is the current
    \begin{displaymath}
        S \times T \in \Curr{i+j}(A \times B)
    \end{displaymath}
    characterised by the following conditions
    \begin{displaymath}
        (S \times T)(p^{\#}\alpha \wedge q^{\#}\beta)
        =
        \left\{
            \begin{aligned}
              &S(\alpha)T(\beta) && \text{if $k=i$}
              \\
              &0 && \text{if $k \neq i$}
            \end{aligned}
        \right.
        \quad \text{for $\alpha \in \Forms{k}(A)$ and $\beta \in \Forms{i+j-k}(B)$} \,.
    \end{displaymath}
\end{definition}

\begin{definition}[\protect{\cite[4.1.11]{Federer1969}}]
    Let $H: \R^{n} \times \R \times \R^{n} \to \R^{n}$ be a map such that
    \begin{displaymath}
        H(x,t,y)=(1-t)x+ty \quad \text{whenever} \quad x,y \in \R^{n} \quad
        \text{and} \quad t \in \R \,.
    \end{displaymath}
    For any currents
    $S \in \Curr{i}(\R^{n})$ and $T \in \Curr{j}(\R^{n})$ with
    compact supports, the \emph{join} of $S$ and $T$ is defined as
    \begin{displaymath}
        S \XX T = H_{\#}(S \times [0,1] \times T) \in \Curr{i+1+j}(\R^{n})
    \end{displaymath}
\end{definition}

\begin{definition}[\protect{\cite[4.1.11]{Federer1969}}]
    Let $u_0,\ldots,u_k \in \R^{n}$. An oriented $k$-simplex
    $\Lbrack u_0,\ldots,u_m \Rbrack$ is defined inductively as
    \begin{gather}
        \Lbrack u \Rbrack = \Dirac{u} \text{ whenever $u \in \R^n$} \,,
        \\
        \Lbrack u_0,\ldots,u_k \Rbrack = \Lbrack u_0 \Rbrack \XX 
        \Lbrack u_1,\ldots,u_k \Rbrack \in \Curr{k}(\R^{n})
        \quad \text{if $k \in \natp$}\,.
    \end{gather}
    If $u_0, \ldots, u_k$ are affinely independent, we say that
    $\Lbrack u_1,\ldots,u_k \Rbrack$ is \emph{non-degenerate}.
\end{definition}

\begin{remark}[\protect{\cite[4.1.11]{Federer1969}}]
    Every oriented $k$-simplex in any Euclidean space is representable as affine
    image of any~\emph{non-degenerate} oriented $k$-simplex in~$\R^{k}$.
\end{remark}

\begin{remark}
    Let $\sigma \subseteq \R^n$ be the convex hull of
    $u_0,\ldots,u_k \in \R^{n}$ and $\tau: \sigma \to \ograss nk$ be the
    orientation of~$\sigma$ defined by
    \begin{displaymath}
        \tau(x) = \frac{(u_1-u_0) \wedge \ldots \wedge (u_k-u_{k-1})}
        {|(u_1-u_0) \wedge \ldots \wedge (u_k-u_{k-1})|}
        \quad \text{for $x \in \sigma$} \,.
    \end{displaymath}
    Then
    \begin{displaymath}
        \Lbrack u_0, \ldots, u_k \Rbrack(\omega)
        = \textint{\sigma}{} \langle \omega(x), \tau(x) 
        \rangle \ud \HM^{k}(x)
        \quad \text{for $\omega \in \Forms{k}(U)$} \,.
    \end{displaymath}
\end{remark}

\begin{definition}[\protect{\cite[4.1.22]{Federer1969}}]
    Let $U$ be an open subset of $\R^{n}$, and $K \subseteq U$ be compact. The
    additive subgroup of~$\Curr{k}(U)$ generated by all oriented
    $k$-simplices $\Lbrack u_0,\ldots,u_k \Rbrack$ such
    that~$\conv \{u_0,\ldots,u_k\} \subseteq K$ is denoted
    \begin{displaymath}
        \IPC{k,K}(U) \,.
    \end{displaymath}
    All the $k$-dimensional \emph{integral polyhedral chains} in~$U$ form the
    abelian group
    \begin{displaymath}
        \IPC{k}(U) = \tbcup \bigl\{ \IPC{k,K}(U) : K \subseteq U \text{ compact} \bigr\} \,.
    \end{displaymath}
\end{definition}

\begin{definition}[\protect{\cite[4.1.22]{Federer1969}}]
    Let $\PC{k, K}(U)$ be the vector-subspace of~$\Curr{k}(U)$ generated
    by $\IPC{k, K}(U)$. The $k$-dimensional \emph{polyhedral chains}
    in~$U$ are members of the vectorspace
    \begin{displaymath}
        \PC{k}(U) = \tbcup \bigl\{ \PC{k, K}(U) : K \subseteq U \text{ compact} \bigr\} \,.
    \end{displaymath}
\end{definition}

\begin{remark}
    Elements of $\PC{k}(U)$ are precisely finite sums of the form
    \begin{displaymath}
        T = \textsum{i=1}{m} a_{i} \Delta_{i}
        \quad \text{where $a_{1}, \ldots, a_m \in \R$
          and $\Delta_{1}, \ldots, \Delta_m$  are oriented $k$-simplices} \,.
    \end{displaymath}
    Furthermore, we may assume $a_i > 0$ for $i \in \{1,2,\ldots,m\}$ since if
    $a_i < 0$ for some $i \in \{1,2,\ldots,m\}$, then we obtain the same $T$ by
    replacing $a_i$ with $-a_i$ and reversing the orientation of
    $\Delta_i$. In~particular, any polyhedral chain has compact support. Now,
    let $\omega \in \Forms{k}(U)$, then
    \begin{displaymath}
        T(\omega) = \textsum{i=1}{m} a_{i} \Delta_{i} (\omega) =
        \textint{\bigcup_{i=1}^{m} \spt \Delta_{i}}{} \langle \omega(x), 
        \textsum{i=1}{m} a_{i} \tau_{i}(x) \rangle \ud \HM^{k}(x) \,;
    \end{displaymath}
    thus, every \emph{polyhedral chain} can be seen as rectifiable $k$-current
    $(E, \tau, \theta)$ corresponding to the countably $k$-rectifiable set
    \begin{displaymath}
        E = \tbcup_{i=1}^{m} \spt \Delta_{i}
    \end{displaymath}
    with orientation given by
    \begin{displaymath}
        \tau(x) = \frac{\tilde{\tau}(x)}{|\tilde{\tau}(x)|}
        \quad \text{for $x \in E$} \,,
        \quad \text{where} \quad
        \tilde{\tau}(x) = \textsum{i=1}{m} a_i \vec{\Delta}_i(x) \CF{\spt \Delta_{i}}(x)
    \end{displaymath}
    and multiplicity function
    \begin{displaymath}
        \theta(x) = |\tilde{\tau}(x)| \,.
    \end{displaymath}
\end{remark}

\subsection*{Gaussian measures}

\begin{definition}[\protect{\cite[Ch.1]{Rosa2023}}]
    Let $A = (E, \tau, \theta)$ be a polyhedral $k$-chain. The \emph{weighted
      Gaussian image} of~$A$ is the measure $\gamma_{A}$ over $\ograss{n}{k}$
    defined by
    \begin{displaymath}
        \gamma_{A} = \tau_{\#}(\theta\HM^{k} \restrict E) \,.
    \end{displaymath}
\end{definition}

\begin{definition}[\protect{\cite[Ch.1]{Rosa2023}}]
    Let $F$ be an integrand. For every polyhedral $k$-chain
    $A=(E, \tau, \theta)$ we define \emph{anisotropic energy functional}
    $\Phi_{F}$ as
    \begin{displaymath}
        \Phi_{F}(A) = \textint{E}{} \theta \cdot F \circ \tau \ud \HM^{k} \,.
    \end{displaymath}
\end{definition}

\begin{remark}
    By the very definition of the push-forward~\cite[2.1.2]{Federer1969} the
    anisotropic energy functional can be written in terms of the weighted
    Gaussian image as follows
    \begin{displaymath}
      \Phi_{F}(A)
      = \textint{E}{} \theta \cdot F \circ \tau \ud\HM^{k} 
      = \textint{\ograss nk}{} F \ud \tau_{\#}( \theta\HM^{k} \restrict E) 
      = \textint{\ograss nk}{} F \ud \gamma_{A} \,.
    \end{displaymath}
\end{remark}

\begin{definition}
    \label{def:mass_of_a_chain}
    The \emph{mass} of $T \in \PC{k}(U)$ is defined as
    \begin{displaymath}
        \Mass(T) = \|T\|(U) \,.
    \end{displaymath}
\end{definition}

\begin{remark}
    In the case when $T = \textsum{i=1}{m} a_{i} \Delta_{i} \in \PC{k}(U)$ for 
    some $a_1, \ldots, a_m \in (0,\infty)$ and oriented $k$-simplices
    $\Delta_1, \ldots, \Delta_m$, we have
    \begin{displaymath}
        \Mass(T)
        = \textsum{i=1}{m} a_{i} \HM^{k}(\Delta_{i})
        = \gamma_T(\ograss nk)
        = \| \gamma_T \|_{\mathrm{TV}}\,.
    \end{displaymath}
\end{remark}

\section{Ellipticity}
\label{section:Ellipticity}

\begin{definition}
    Recalling~\cite[4.1.8 and 4.1.32]{Federer1969} for $k \in \natp$ we set
    \begin{displaymath}
        Q_k = \underbrace{\Lbrack 0, 1 \Rbrack \times \cdots \times \Lbrack 0, 1 \Rbrack}_{\text{$k$ factors}}
        \in \IPC{k}(\R^k) \,.
    \end{displaymath}
\end{definition}

\begin{definition}[\protect{cf.~\cite[4.1]{Rosa2020}}]
    \label{def:polyhedral_test_pairs}
    We~say that $(S, D) \in \mathcal{P}$ is a~\emph{polyhedral test pair} if
    there exists $p \in \orthproj nk$ such that
    \begin{gather*}
        S \in \PC{k}(U) \,,
        \quad
        D = (p^*)_{\#} Q_k \,,
        \quad \text{and} \quad
        \partial S = \partial D  \,.
    \end{gather*}
\end{definition}

\begin{definition}
    \label{def:ellipticity}
    Let $F$ be an integrand and $\mathcal{P}$ be the set of polyhedral test
    pairs.  We say that $F$ is \emph{elliptic with respect to $\mathcal{P}$} if
    there exists $0 < c < \infty$ (the \emph{ellipticity constant}) such that
    \begin{displaymath}
        \Phi_F(S) - \Phi_F(D) > c \bigl( \Mass(S) - \Mass(D) \bigr)
        \quad \text{for $(S,D) \in \mathcal{P}$} \,.
    \end{displaymath}
    If this holds we write $F \in \mathrm{AUE}(\mathcal{P}, c)$.
\end{definition}

\subsection*{Main result}

\begin{miniremark}
    Our main result depends on the theorem proven recently by De~Rosa, Lei, and~Young.
    Before stating it we first discuss some definitions.
\end{miniremark}

\begin{definition}[\protect{cf.~\cite[\S{3}]{Rosa2023}}]
    Let $e_1,\ldots,e_{n}$ be the standard basis of $\R^{n}$ and $P \subseteq \R^n$ be
    a~linear subspace of dimension~$k$. We say that~$P$ has \emph{rational slope} if there
    exists $f \in \Hom(\R^{n},\R^{n-k})$ such that $P = \ker f$ and
    $f(e_i) \in \integers^{n-k}$ for $i \in \{ 1,2,\ldots,n \}$.
\end{definition}

\begin{definition}
    Let $u_1,\ldots,u_{n-k}$ be the standard basis of~$\R^{n-k}$ and
    $\ast : \tbwedge_{n-k} \R^n \to \tbwedge_k \R^n$ be the Hodge star defined with
    respect to the standard orientation of $\R^n$; cf.~\cite[1.7.8]{Federer1969}.
    Set
    \begin{gather}
        \xi_{f} = \ast \bigl( f^*(u_1) \wedge \cdots \wedge f^*(u_{n-k}) \bigr)
        \in \tbwedge_k \R^n
        \quad \text{for $f \in \Hom(\R^n,\R^{n-k})$} 
        \\
        \text{and} \quad
        Q(n,k) = \bigl\{ t \xi_f :
        t \in \R ,\,
        f \in \Hom(\R^n,\R^{n-k}) ,\, f^*(u_i) \in \integers^{n}
        \text{ for } i \in \{ 1,2,\ldots,n-k \}
        \bigr\} \,.
    \end{gather}
\end{definition}

\begin{remark}
    Let $e_1,\ldots,e_{n}$ be the standard basis of~$\R^{n}$ and $u_1,\ldots,u_{n-k}$ be
    the standard basis of~$\R^{n-k}$. If $P \subseteq \R^n$ is a $k$-dimensional linear
    space with rational slope and $f \in \Hom(\R^n,\R^{n-k})$ is such that $P = \ker f$
    and $f(e_i) \in \integers^{n-k}$ for $i \in \{ 1,2,\ldots,n \}$, then
    $\xi_f \in Q(n,k)$.  On the other hand, if $\xi \in Q(n,k)$, then there is
    $f \in \Hom(\R^n,\R^{n-k})$ such that $f(e_i) \in \integers^{n-k}$ for
    $i \in \{ 1,2,\ldots,n \}$ and
    \begin{displaymath}
        \ker f = (\im f^*)^{\perp} = \asssp \xi \,;
    \end{displaymath}
    hence, $\asssp \xi$ has rational slope.
\end{remark}

\begin{lemma}
    \label{lem:SQ_dense}
    There holds
    \begin{displaymath}
        S_{\Q} = \bigl\{ w_1 \wedge \cdots \wedge w_k : w_1,\ldots,w_k \in \Q^k \bigr\}
        \subseteq Q(n,k)
        \subseteq S_{\R} = \tbwedge_k \R^n \cap \{ \eta : \eta \text{ is simple} \} 
    \end{displaymath}
    and $S_{\Q}$ is dense in $S_{\R}$.
\end{lemma}

\begin{proof}
    Let $e_1,\ldots,e_{n}$ be the standard basis of~$\R^{n}$ and $u_1,\ldots,u_{n-k}$ be
    the standard basis of~$\R^{n-k}$. We first show that $S_{\Q}$ is dense in~$S_{\R}$.
    Suppose $0 < \varepsilon < 1$ and $\eta = v_1 \wedge \cdots \wedge v_k$ for some
    $v_1, \ldots, v_k \in \R^n$. Without loss of generality assume $|v_i| = 1$ for
    $i \in \{1,2,\ldots,k\}$. Since $\Q^n$ is dense in~$\R^n$, there exist
    $w_1, \ldots, w_k \in \Q^n$ such that $|w_i - v_i| \le 2^{-k} \varepsilon$ for
    $i \in \{1,2,\ldots,k\}$. Setting $\zeta = w_1 \wedge \cdots \wedge w_k$ and
    $r_i = w_i - v_i$ for $i \in \{1,2,\ldots,k\}$ we get
    \begin{displaymath}
        |\zeta - \eta|
        = |(v_1+r_1) \wedge \cdots \wedge (v_k + r_k) - v_1 \wedge \cdots \wedge v_k|
        \le \textsum{i=1}{k} {\textstyle \binom{k}{i}} \varepsilon^i
        \le \varepsilon \,.
    \end{displaymath}

    Next, we show that $S_{\Q} \subseteq Q(n,k)$. Assume $w_1, \ldots, w_n \in \Q^n$ and
    $(w_1 \wedge \cdots \wedge w_n) \bullet (e_1 \wedge \cdots \wedge e_n) = s \ne 0$.
    Set $\zeta = w_1 \wedge \cdots \wedge w_k$. We shall find $t \in \R$ and
    $f \in \Hom(\R^n,\R^k)$ satisfying $f(e_i) \in \integers^n$ for
    $i \in \{1,2,\ldots,k\}$ and such that $\zeta = t \xi_f$. Define
    \begin{displaymath}
        y_1 = w_1 \,,
        \quad
        y_j = w_j - \textsum{i=1}{j-1} |y_i|^{-2} (w_j \bullet y_i) y_i
        \quad \text{for $j \in \{2,3,\ldots,n\}$} \,.
    \end{displaymath}
    Note that
    \begin{gather}
        |y_i|^2 \in \Q \,,
        \quad
        y_i \in \Q^n \,,
        \quad
        y_i \perp y_j \,,
        \\
        \quad \text{and} \quad
         y_1 \wedge \cdots \wedge y_i = w_1 \wedge \cdots \wedge w_i
        \quad \text{for $i,j \in \{1,2,\ldots,n\}$ with $i \ne j$} \,.
    \end{gather}
    Clearly there is $M \in \natp$ such that $M y_i \in \integers^n$ for
    $i \in \{1,2,\ldots,n\}$.  Define $f \in \Hom(\R^n,\R^{n-k})$ by~requiring
    \begin{displaymath}
        f^{*}(u_j) = M y_{k+j}
        \quad \text{for $j \in \{1,2,\ldots,n-k\}$} \,.
    \end{displaymath}
    Then, recalling~\cite[1.7.8 and~1.5.2]{Federer1969} and
    $y_1 \wedge \cdots \wedge y_n = s e_1 \wedge \cdots \wedge e_n$, we get
    \begin{gather}
        f(e_i) \in \integers^{n-k} 
        \quad \text{for $i \in \{1,2,\ldots,n\}$} \,,
        \\
        \xi_{f}
        = \ast \bigl( My_{k+1} \wedge \cdots \wedge My_{n}  \bigr)
        = M^{n-k} (-1)^{k(n-k)} \frac{|y_{k+1}| \cdots |y_{n}|}{|y_{1}| \cdots |y_{k}|} y_{1} \wedge \cdots \wedge y_{k}
        = t^{-1} \zeta \,,
    \end{gather}
    where $t = M^{k-n} (-1)^{k(n-k)} |y_{1}| \cdots |y_{k}| \bigl( |y_{k+1}| \cdots |y_{n}| \bigr)^{-1} \ne 0$.
\end{proof}

\begin{theorem}[\protect{cf.~\cite[3.2]{Rosa2023}}]
    \label{thm:antonio}
    Assume
    \begin{gather}
        \eta_{0} \in \ograss nk \,,
        \quad
        \eta_{1}, \ldots, \eta_{d} \in Q(n,k) \cap \ograss nk \,,
        \quad
        m_{1},\ldots,m_{d} \in (0,\infty) \,,
        \\
        \textsum{i=1}{d} m_{i}\eta_{i}=\eta_{0} \,,
        \quad
        E \in \ograss kk \text{ corresponds to the standard orientation of~$\R^k$} \,,
        \\
        p \in \orthproj nk \text{ is such that } \tbwedge_k p^*(E) = \eta_0 \,,
        \quad
        D = (p^*)_{\#}Q_k \,,
        \quad \text{and} \quad
        \mu = \textsum{i=1}{d} m_{i} \Dirac{\eta_{i}} \,.
    \end{gather}
    Then for any $\epsilon > 0$, there exists $A \in
    \PC{k}(\R^{n})$ with coefficients in~$\Q$ such that
    \begin{displaymath}
        \partial A = \partial D
        \quad \text{and} \quad
        \|\gamma_{A} - \mu\|_{\mathrm{TV}} < \epsilon \,.
    \end{displaymath}
\end{theorem}

\begin{theorem}
    \label{thm:main}
    Assume $F$ is a Lipschitzian integrand and $0 < c_{2} < c_{1} < \infty$. If
    $F \in \mathrm{AUE}(\mathcal{P}, c_{1})$, then $F \in \mathrm{UPC(c_{2})}$.
\end{theorem}

\begin{proof}
    We argue by contradiction. Assume $F \in \mathrm{AUE}(\mathcal{P}, c_{1})$ is
    Lipschitzian with $\Lip F = L < \infty$ and $F \notin \mathrm{UPC(c_{2})}$. Use the
    Kirszbraun theorem~\cite[2.10.43]{Federer1969} to extend~$F$ to the whole
    of~$\tbwedge_k \R^n$ preserving the Lipschitz constant. We shall use the same
    symbol~$F$ to denote this extension. Since $F \notin \mathrm{UPC(c_{2})}$, there exist
    $d \in \natp$, $\eta_{0},\ldots,\eta_{d} \in \ograss{k}{n}$, and
    $m_1, \ldots, m_d \in (0,\infty)$ such that
    \begin{equation}
        \label{eq:F_not_UPC}
        \textsum{i=1}{d}m_{i}\eta_{i} = \eta_{0}
        \quad \text{and} \quad
        \textsum{i=1}{d} m_{i} F(\eta_{i}) - F(\eta_{0})
        < c_{2} \bigl( \textsum{i=1}{d} m_{i} |\eta_{i}| - |\eta_{0}| \bigr) \,.
    \end{equation}
    Since $|\eta_i| = 1$ for $i \in \{1,2,\ldots,d\}$,
    $\bigl| \sum_{i=1}^{d} m_i \eta_i \bigr| = 1$, and the Euclidean norm on
    $\tbwedge_k \R^n$ is strictly convex, we see that $d \ge 3$ and
    $\sum_{i=1}^{d} m_i > 1$. Let
    \begin{displaymath}
        M = \sup \im F
        \quad \text{and} \quad
        \epsilon = \frac{\frac 12 (c_{1}-c_{2})\big(\sum_{i=1}^{d}m_{i} - 1\big)}
        {\frac 32 L + M + M \binom nk^{\frac 12} + c_{1}} \,.
    \end{displaymath}
    For each $i \in \{0, 1,\ldots,d\}$ we choose, using~\ref{lem:SQ_dense}, a~unit simple
    $k$-vector $\tilde{\eta}_{i} \in Q(n,k) \cap \ograss nk$ satisfying
    \begin{equation}
        \label{eq:choice_of_eta}
        |\eta_{i} - \tilde{\eta_{i}}| < \frac{\epsilon}{2 d m_{i}}
        \quad \text{if $i \ge 1$}
        \quad \text{and} \quad
        |\eta_{0} - \tilde{\eta_{0}}| < \frac{\epsilon}{2} \,;
    \end{equation}
    next, we set
    \begin{displaymath}
        \zeta = \tilde{\eta_{0}} - \textsum{i=1}{d} m_{i} \tilde{\eta_{i}} \,.
    \end{displaymath}
    Since $|\eta_i| = 1 = |\tilde{\eta}_i|$ for $i \in \{ 0,1,\ldots,d \}$ we~obtain
    \begin{align}
      \label{eq:estiamte_for_zeta}
      |\zeta| = 
      \bigl| \tilde{\eta_{0}} - \textsum{i=1}{d} m_{i}\tilde{\eta_{i}} \bigr|
      &\leq \bigl| \tilde{\eta_{0}} - \eta_{0} \bigr|
        +\bigl| \eta_{0} - \textsum{i=1}{d} m_{i}\tilde{\eta_{i}} \bigr|
      \\
      &= \bigl| \tilde{\eta_{0}} - \eta_{0} \bigr|
        + \bigl| \textsum{i=1}{d}m_{i}\eta_{i}
        - \textsum{i=1}{d} m_{i}\tilde{\eta_{i}} \bigr|
      \\
      &\leq \bigl| \tilde{\eta_{0}} - \eta_{0} \bigr|
        + \textsum{i=1}{d} m_{i} \bigl| \eta_{i} - \tilde{\eta_{i}} \bigr|
      \\
      &< \frac{\epsilon}{2}+ \frac{\epsilon}{2} = \epsilon \,.
    \end{align}
    Let $e_1, \ldots, e_n$ be the standard basis of~$\Q^n \subseteq \R^n$ and
    $e_{\lambda} = e_{\lambda(1)} \wedge \cdots \wedge e_{\lambda(k)}$ for
    $\lambda \in \Lambda(n,k)$. Recalling~\cite[1.7.5]{Federer1969} we see that
    $\{ e_{\lambda} : \lambda \in \Lambda(n,k) \}$ is an orthonormal basis
    of~$\tbwedge_k \R^n$. Set
    \begin{gather}
        m_{\lambda} = |\zeta \bullet e_{\lambda}| 
        \quad \text{and} \quad
        \eta_{\lambda} = \sgn( \zeta \bullet e_{\lambda} ) e_{\lambda}
        \quad \text{for $\lambda \in \Lambda(n,k)$} \,,
        \\
        L = \Lambda(n,k) \cap \{ \lambda : m_{\lambda} > 0 \} \,.
    \end{gather}
    We get
    \begin{displaymath}
        \tilde{\eta_{0}} = \textsum{i=1}{d} m_{i} \tilde{\eta_{i}}
        + \textsum{\lambda \in L}{} m_{\lambda} \eta_{\lambda} \,.
    \end{displaymath}
    Denote
    \begin{displaymath}
        \mu = \textsum{i=1}{d}m_{i} \Dirac{\eta_{i}}
        \quad \text{and} \quad
        \tilde{\mu} = \textsum{i=1}{d} m_{i}\Dirac{\tilde{\eta}_{i}}
        + \textsum{\lambda \in L}{} m_{\lambda}\Dirac{\eta_{\lambda}} \,.
    \end{displaymath}
    Clearly $\eta_{\lambda} \in Q(n,k) \cap \ograss nk$ and $m_{\lambda} > 0$ for
    $\lambda \in L$. Therefore, the set of unit simple $k$-vectors $\{ \tilde{\eta}_{0}$,
    $\tilde{\eta}_{1}$, \ldots,
    $\tilde{\eta}_{d} \} \cup \{ \eta_{\lambda} : \lambda \in L \}$, the numbers
    $\{ m_1, \ldots, m_d \} \cup \{ m_{\lambda} : \lambda \in L \}$, and the measure
    $\tilde{\mu}$ satisfy the assumptions of theorem~\ref{thm:antonio}. Thus, there exist
    $k$-polyhedral chains $A, \tilde{D} \in \PC{k}(\R^{n})$ such that
    \begin{displaymath}
        \vec{D} = \tilde{\eta}_0 \,,
        \quad
        \spt D \text{ is a unit cube in $\asssp \tilde{\eta}_0$} \,,
        \quad
        \partial A = \partial \tilde{D} \,,
        \quad \text{and} \quad
        \|\gamma_{A} - \tilde{\mu}\|_{\mathrm{TV}} < \epsilon \,.
    \end{displaymath}
    By~the assumption that $F \in \mathrm{AUE}(\mathcal{P}, c_{1})$ we have the
    following
    \begin{align}
      \Phi(A)-\Phi(\tilde{D})
      &= \textint{\ograss{n}{k}}{} F \ud \gamma_{A} -
        \textint{\ograss{n}{k}}{} F \ud \Dirac{\tilde{\eta}_{0}} \\
      &= \textint{\ograss{n}{k}}{} F \ud (\gamma_{A}-\tilde{\mu}) \label{est:1} \\
      &\phantom{= } + \textint{\ograss{n}{k}}{} F \ud(\tilde{\mu}-\mu) \label{est:2} \\
      &\phantom{= } + \textint{\ograss{n}{k}}{} F \ud \mu - F(\tilde{\eta}_{0}) \,. \label{est:3}
    \end{align}
    We will show separately upper bounds for \eqref{est:1}, \eqref{est:2} and
    \eqref{est:3}.

    \smallskip
    
    \noindent
    \fbox{Estimate of (\ref{est:1})} Obviously
    \begin{align}
      \textint{\ograss{n}{k}}{} F \ud (\gamma_{A}-\tilde{\mu})
      \leq M \|\gamma_{A}-\tilde{\mu}\|_{\mathrm{TV}}
      < M \epsilon \,.
    \end{align}

    \smallskip
    
    \noindent
    \fbox{Estimate of (\ref{est:2})} Recalling~\eqref{eq:choice_of_eta}
    and~\eqref{eq:estiamte_for_zeta} we get
    \begin{align}
      \textint{\ograss{n}{k}}{} F \ud (\tilde{\mu}-\mu)
      &= \textint{\ograss{n}{k}}{} F \ud \tilde{\mu}
        - \textint{\ograss{n}{k}}{} F \ud {\mu}
      \\
      &= \textsum{i=1}{d}m_{i}F(\tilde{\eta}_{i}) +
        \textsum{\lambda \in L}{}m_{\lambda}F(\eta_{\lambda}) -
        \textsum{i=1}{d}m_{i}F(\eta_{i})
      \\
      &= \textsum{i=1}{d} \bigl( m_{i}F(\tilde{\eta}_{i})-m_{i}F(\eta_{i}) \bigr) +
        \textsum{\lambda \in L}{}m_{\lambda}F(\eta_{\lambda})
      \\
      &\le L \textsum{i=1}{d} m_{i} |\eta_{i} - \tilde{\eta}_i| +
        M \textsum{\lambda \in L}{} m_{\lambda}
      \\
      &\leq L\textsum{i=1}{d}|m_{i}||\tilde{\eta_{i}} - \eta_{i}| +
        M \bigl( \textsum{\lambda \in L}{}m_{\lambda}^{2} \bigr)^{\frac{1}{2}}(\card L)^{\frac{1}{2}}
      \\
      &< \tfrac 12 L \epsilon +
        M \bigl( \textsum{\lambda \in L}{} |m_{\lambda} \eta_{\lambda} |^{2} \bigr)^{\frac{1}{2}}
        {\textstyle \binom nk}^{\frac{1}{2}}
      \\
      &= \tfrac 12 L \epsilon + M {\textstyle \binom nk}^{\frac{1}{2}} |\zeta|
        % \bigl| \textsum{\lambda \in L}{}m_{i} \eta_{\lambda} \bigr|
        < \bigl( \tfrac 12 L + M {\textstyle \binom nk}^{\frac{1}{2}} \bigr) \epsilon \,,
    \end{align}
    where the next to last inequality is a consequence of
    $\{ \eta_{\lambda} : \lambda \in \Lambda(n,k) \}$ being an orthonormal basis and the
    Pythagorean Theorem.

    \smallskip
    
    \noindent
    \fbox{Estimate of (\ref{est:3})} Recalling~\eqref{eq:F_not_UPC} we get
    \begin{align}
      \textint{\ograss{n}{k}}{} F \ud \mu - F(\tilde{\eta}_{0})
      &= \textsum{i=1}{d} m_{i} F(\eta_{i}) - F(\eta_{0}) + (F(\eta_{0}) - F(\tilde{\eta_{0}})) \\
      &< c_{2} \bigl( \textsum{i=1}{d}m_{i}|\eta_{i}|- |\eta_{0}| \bigr) + L |\eta_{0} - \tilde{\eta_{0}}| \\
      &< c_{2} \bigl( \textsum{i=1}{d}m_{i}- 1 \bigr) + L\epsilon \,.
    \end{align}
    
    Now, combining all the estimates~\eqref{est:1}, \eqref{est:2}, and~\eqref{est:3} gives
    \begin{equation}
        \label{est:upper}
        \Phi(A) - \Phi(\tilde{D}) <
        c_{2}\big( \textsum{i=1}{d}m_{i}- 1 \big) +
        \epsilon \big( \tfrac 32 L + M + M {\textstyle \binom nk}^{\frac{1}{2}} \big)  \,.
    \end{equation}
    On the other hand, by~\cite[2.3]{Rosa2023} and recalling that $\spt D$ is a unit cube,
    we have the following lower estimate
    \begin{equation}
        \label{eq:Mass_and_TV}
        c_{1} \bigl( \Mass(A) - \Mass(\tilde{D}) \bigr)
        \ge c_{1} \bigl( \|\gamma_{A}\|_{\mathrm{TV}} - 1 \bigr) \,.
    \end{equation}
    Moreover, we observe
    \begin{displaymath}
        \textsum{i=1}{d} m_{i}
        \leq \textsum{i=1}{d} m_{i} + \textsum{\lambda \in L}{} m_{\lambda} 
        = \|\tilde{\mu}\|_{\mathrm{TV}}
        \le \|\tilde{\mu} - \gamma_{A}\|_{\mathrm{TV}} +
        \|\gamma_{A}\|_{\mathrm{TV}} 
        < \epsilon + \|\gamma_{A}\|_{\mathrm{TV}} \,;
    \end{displaymath}
    thus,
    \begin{equation}
        \label{est:lower}
        c_{1}\big( \|\gamma_{A}\|_{\mathrm{TV}} - 1 \big) =
        c_{1}\big( \|\gamma_{A}\|_{\mathrm{TV}} + \epsilon - \epsilon - 1 \big)
        > c_{1} \bigl( \textsum{i=1}{d}m_{i} - 1 \bigr) - c_{1}\epsilon \,.
    \end{equation}
    Finally, since $(A,\tilde{D})$ is a polyhedral test pair and
    $F \in \mathrm{AUE}(\mathcal{P}, c_{1})$, combining~\eqref{est:upper},
    \eqref{eq:Mass_and_TV}, and~\eqref{est:lower} we~obtain
    \begin{multline}
        c_{2}\bigl( \textsum{i=1}{d}m_{i}- 1 \bigr) + 
        \epsilon \bigl( \tfrac 32 L + M + M{\textstyle \binom nk}^{\frac{1}{2}} \bigr) \\
        > \Phi(A) - \Phi(\tilde{D})
        \ge c_{1} \bigl( \Mass(A) - \Mass(\tilde{D}) \bigr) 
        > c_{1} \bigl( \textsum{i=1}{d}m_{i} - 1 \bigr) - c_{1}\epsilon \,.
    \end{multline}
    Recalling the definition of $\epsilon$ we see that this cannot happen.
\end{proof}

\begin{remark}
    \label{rem:C1_integrands}
    Note that if $F$ is of class~$\cnt{1}$, then standard tools
    (see~\cite[3.1.14]{Federer1969}) allow to extend $F$ to a~$\cnt{1}$ function on some
    compact neighbourhood of~$\ograss nk$ and then it is clear that $F$ is Lipschitzian so
    our theorem applies to any integrand of class~$\cnt{1}$.
\end{remark}

\subsection*{Acknowledgements}
This research has been supported by \href{https://ncn.gov.pl/}{National Science Centre
  Poland} grant number 2022/46/E/ST1/00328. The author would like to extend sincere
gratitude to Sławomir Kolasiński for his invaluable assistance and insightful discussions
throughout this research.

\subsection*{Data Availability Statement}
Data sharing is not applicable to this article as no datasets were generated or analysed during the research.

\bigskip

{\small
\bibliographystyle{halpha}
\addcontentsline{toc}{section}{\numberline{}References}
\bibliography{auepolyconvex.bib}{}
}

\bigskip

{
\small \noindent
Maciej Lesniak
\\
\texttt{maciej.lesniak0@gmail.com}
}

\end{document}